\documentclass[10pt]{amsart}
\usepackage[margin=1.2in,marginparsep=0.1in,marginparwidth=1in]{geometry}
\usepackage{amssymb,amsmath,amsthm,amstext,amscd,latexsym,graphics,graphicx,bbm,caption}
\usepackage[usenames,dvipsnames,svgnames,table]{xcolor}
\usepackage[plainpages=false,colorlinks=true, pagebackref]{hyperref}
\usepackage{tikz}
\usepackage{tikz-cd}
\usetikzlibrary{positioning}

\tikzset{>=stealth}
\hypersetup{citecolor=Sepia,linkcolor=blue, urlcolor=blue}

\usepackage{cite}   

\makeatletter
\def\@tocline#1#2#3#4#5#6#7{\relax
 \ifnum #1>\c@tocdepth 
 \else
  \par \addpenalty\@secpenalty\addvspace{#2}%
  \begingroup \hyphenpenalty\@M
  \@ifempty{#4}{%
   \@tempdima\csname r@tocindent\number#1\endcsname\relax
  }{%
   \@tempdima#4\relax
  }%
  \parindent\z@ \leftskip#3\relax \advance\leftskip\@tempdima\relax
  \rightskip\@pnumwidth plus4em \parfillskip-\@pnumwidth
  #5\leavevmode\hskip-\@tempdima
   \ifcase #1
    \or\or \hskip 2em \or \hskip 2em \else \hskip 3em \fi%
   #6\nobreak\relax
  \dotfill\hbox to\@pnumwidth{\@tocpagenum{#7}}\par
  \nobreak
  \endgroup
 \fi}
\makeatother

\newtheorem{intro-thm}{Theorem}[]
\theoremstyle{plain}
\newtheorem{thm}{Theorem}[section]
\newtheorem{theorem}[thm]{Theorem}

\newtheorem{lemma}[thm]{Lemma}

\newtheorem{proposition}[thm]{Proposition}

\theoremstyle{definition}
\newtheorem{remark}[thm]{Remark}

\newtheorem{definition}[thm]{Definition}
\newtheorem{example}[thm]{Example}

\newcommand{\intersection}{\cap}

\newcommand{\Proj}{{\rm Proj \,}}

\newcommand{\Spec}{{\rm Spec \,}}

\renewcommand{\tilde}{\widetilde}

\newcommand{\sD}{{\mathcal D}}
\newcommand{\sE}{{\mathcal E}}

\newcommand{\sO}{{\mathcal O}}

\newcommand{\A}{{\mathbb A}}

\renewcommand{\P}{{\mathbb P}}

\newcommand{\etale}{\'{e}tale}

\newcommand{\ra}{\rightarrow}

\input{xy}
\xyoption{all}

\begin{document}

\title{Nisnevich local Good compactifications}

\author{Neeraj Deshmukh}
\address{Institut f\"{u}r Mathematik, Universit\"{a}t Z\"{u}rich, Winterthurerstrasse 190, CH-8057 Z\"{u}rich, Switzerland}
\email{neeraj.deshmukh@math.uzh.ch}

\author{Amit Hogadi}
\address{Department of Mathematics, Indian Insititute of Science Education and Research (IISER) Pune, Dr. Homi Bhabha road, Pashan, Pune 411008 India}
\email{amit@iiserpune.ac.in}

\author{Girish Kulkarni}
\address{Fachgruppe Mathematik/Informatik, Bergische Universität Wuppertal, Gaußstraße 20, 42119 Wuppertal, Germany}
\email{kulkarni@uni-wuppertal.de}

\author{Suraj Yadav}
\address{Department of Mathematics, Indian Insititute of Science Education and Research (IISER) Pune, Dr. Homi Bhabha road, Pashan, Pune 411008 India}
\email{surajprakash.yadav@students.iiserpune.ac.in}



\date{}

\begin{abstract} 
	For a local complete intersection morphism, we establish fibrewise denseness in the $n$-dimensional irreducible components of the compactification Nisnevich locally.  
\end{abstract}

\maketitle

\section{Introduction}

Let $X$ be a scheme over a base $B$. In this short note, we explore certain compactifications of $X$ in the Nisnevich topology. Roughly speaking, by a \textit{compactification} we mean a quasi-compact scheme  $\overline{X}$ with an open immersion $i: X\hookrightarrow \overline{X}$ over $B$. We also want $\overline{X}$ to have additional nice properties, like projectivity, or that the open immersion $i: X\hookrightarrow \overline{X}$ behaves well with respect to fibres over $B$ (See Theorem \ref{ydense} for the precise conditions that we impose). The existence of such ``good compactifications"\footnote{This terminology is borrowed from \cite[\S 10]{levine2006}.} for $X$ is a useful technical tool for intersection theory and $\A^1$-algebraic topology. 

For instance, such compactifications can be employed to prove moving lemmas for algebraic cycles (see \cite{levine2006}, \cite{kai2015}). Moreover, they are useful for constructing finite morphisms from $X$ to $\A_B^n$ or $\P_B^n$. The existence of such finite maps to $\A_B^n$ is very useful, and has recently been used to prove Gabber presentation lemma, and as a consequence Shifted $\A^1$-Connectivity, over Dedekind domains with infinite residue fields \cite{schmidt2018} and later, more generally, over Noetherian domains with infinite residue fields \cite{deshmukh}. 

However, as noted in \cite[Remark 5.1.1(2)]{levine2006}, one can construct examples of schemes which do not admit any finite surjective map to $\A_B^n$ (see also \cite[\S 8]{gabberhypersurfaces}). But such finite maps always exist \textit{Nisnevich locally}. 
Such Nisnevich local compactifications have been constructed in \cite{levine2006}, \cite{kai2015}, and \cite{deshmukh} in different settings. 

The good compactification constructed in \cite[Theorem 4.1]{kai2015} is over a Dedekind domain $B$ with infinite residue fields. This has been used in \cite{schmidt2018} to prove Shifted $\A^1$-connectivity over Dedekind domains with infinite residue fields. In \cite{deshmukh} a weaker form of Kai's result is proved for Noetherian domains with infinite residue fields which turns out to be sufficient for proving Shifted $\A^1$-connectivity.

In this note, we explore the good compactification constructed in \cite[Theorem 2.1]{deshmukh}. We attempt to extend it to the arithmetic setting by removing the hypothesis on the residue fields. Additionally, we do not assume the base to be reduced. We also observe that the proof in \cite{deshmukh} works for any local complete intersection scheme (see Definition \ref{definition-lci-at-point}). \\

\noindent \textit{Notations and Conventions}: Let $S= \Spec R$. Whenever the embedding of $\A^n_S = \Spec R[x_1, \cdots x_n]$ into $\P^n_S = \Proj R[x_0, \cdots x_{n}]$  is considered without any specification, it is the  open embedding given by the complement of hyperplane $V_+(x_{0})$. Accordingly,  given any subscheme of $Y$ of $\A^n$, by projective closure -- denoted $\overline{Y}$ -- we will mean its closure in $\P^n$ via the open embedding described in previous sentence.
\begin{definition}\cite{nisnevich1989}
A Nisnevich neighbourhood of $(X,x)$ is a pair $(X',x')$ together with \etale{} morphism $f \colon X' \to X$ such that $f(x')=x$ and the induced morphism of residue fields $k(x) \to k(x')$ is an isomorphism. We will frequently abuse notation and denote the pre-image $x'$ by $x$.
\end{definition}

Given a morphism of schemes $f \colon X \rightarrow Y$  and $y \in Y$, $X_y$ will denote the fibre over $y$.

\begin{theorem}\label{ydense}
	Let $B$ be the spectrum of a Noetherian ring of finite dimension. Let $g \colon X\rightarrow B$ be a finite type scheme over $B$. Let $x\in X$ be a point lying over a point $b\in B$ such that $g$ is a local complete intersection at $x$ and $\dim(X_b) = n$. Then there exist Nisnevich neighbourhoods $(X',x)\ra(X,x)$ and $(B',b)\ra(B,b)$, fitting into the following commutative diagram 
	\begin{center}
		\begin{tikzcd}
		& X' \arrow[d] \arrow[r]&X\arrow{d}\\
		& B'\arrow[r]&{B}
		\end{tikzcd}
	\end{center}
	and a closed immersion $X'\ra \A_{B'}^N$ for some $N \geq 0$ such that if $\overline{X'}$ is its closure in $\P_{B'}^N$ then $X'_b$ is dense in the union of $n$-dimensional irreducible components of $(\overline{X'})_b$. 
\end{theorem}

The above theorem is different from \cite[Theorem 2.1]{deshmukh} in the respect that the hypothesis on the residue fields is relaxed, and we also extend the argument to local complete intersection morphisms.

\begin{remark}\label{Kaiapp}
	Theorem \ref{ydense} compares with \cite[Theorem 4.1]{kai2015} as follows:
	\begin{enumerate}
		\item 	The above theorem is a weaker statement than \cite[Theorem 4.1]{kai2015} (see also \cite[Theorem 10.2.2]{levine2006}) in the sense that we prove the denseness of the fibre only in the $n$-dimensional components whereas in \cite{kai2015} $X'_b$ is proved to be dense in $(\overline{X'})_b$. 
		\item On the other hand, \cite[Theorem 4.1]{kai2015} only deals with base a Dedekind domain while Theorem \ref{ydense} is proved for Noetherian rings of finite dimension. Further, there is no restriction on the residue fields of the base to be infinite.
		\item Similar arguments as in the proof of Theorem \ref{ydense} should also generalise \cite[Theorem 4.1]{kai2015} to all Dedekind domains (without any conditions on the residue fields).
	\end{enumerate}
\end{remark}

Just as \cite[Theorem 2.1]{deshmukh}, the proof of Theorem \ref{ydense} is similar to the proof of \cite[Theorem 4.1]{kai2015}. The new ingredient is the observation that over finite fields one can replace arguments involving generic hyperplanes with generic hypersurfaces via the Veronese embedding.

We end the paper with an application of Theorem \ref{ydense} to Nisnevich local finite maps to $\A^n$.\\

\noindent\textbf{Data availability statement.} Data sharing not applicable to this article as no datasets were generated or analysed during the current study.\\

\noindent\textbf{Acknowledgements.} The first-named author was supported by the INSPIRE fellowship (IF160348) of the Department of Science and Technology, Govt.\ of India during the course of this work. He was also partially supported by the Swiss National Science Foundation (SNF), project 200020\_178729. The third-named author was supported by DFG SPP 1786 grant for his stay at Bergische Universit\"at Wuppertal. The last-named author was supported by NBHM fellowship of the Department of Atomic Energy, Govt.\ of India during the course of this work. We thank the anonymous referee for their comments and suggestions.

\section{Proof of Theorem}
\label{section-proof}

In this section we prove our result. Before we begin with the proof, we would like to remark that it is very easy to construct examples of schemes for which the fibre $X_b$ is not dense in the fibre $(\overline{X})_b$ of its projective closure, as can be seen from the following example.

\begin{example}
	Consider the affine scheme $X =\Spec (k[t][X,Y]/(tXY+X+1))$ over $\Spec k[t]$. Homogenising followed by substituting $t=0$, we get 
	\[(tXY+X+1)\rightsquigarrow tXY+XZ+Z^2\rightsquigarrow (X+Z)Z=(\overline{X})_0.\]
	On the other hand, reversing the order of these operations gives us,
	\[(tXY+X+1)\rightsquigarrow X+1\rightsquigarrow (X+Z)=\overline{X_0}\]
	Thus, over $t=0$, the fibre $X_0$ of $X$ is not dense in the fibre $(\overline{X})_0$ of its projective closure. That is to say that the projective closure $\overline{X}$ ``acquires components at infinity".
\end{example}

The next example shows that fibre dimension can increase  after taking the projective closure.
\begin{example}
	Let $S = \A^2_k =  \Spec k [x,y]$ and $X = \P^2_S = \Proj R [t,w,v] $, where $R = k[x,y]$. Let $x = (0,0,0,1,0) \in X$ lying above $s = (0,0) \in S$ and $Y$ be the blowup of $X$ at $x$ . As $Y$ is projective (over $S$), consider a closed embedding $Y \hookrightarrow \P^N_S$ for some $N$. $Y_s$ contains the exceptional divisor which is $\P^3_k$, hence its dimension is 3. We will show that there exists a subscheme of $\A^N_S$ all of whose fibres have dimension at most $2$ and whose projective closure in $\P^N_S$ is $ Y$.\\
	Take the point $x' = (0,0,0,0,1)$ and an irreducible affine neighbourhood $X'\subseteq X$ around it not containing $x$. In fact, we can take $X'$ to be the complement of hyperplane $V_+(v) $, which is isomorphic to $\A^2_S$. Then $X'_s$ is isomorphic to $\A^2_k$, therefore of dimension $2$. Moreover as $X'$ does not contain $x$, there is an open affine in $Y$, say $Z$ which is isomorphic to $X'$ via the blowdown map and a point $z$ in $Z$ mapping to $x'$. As the blowup of an irreducible scheme is irreducible, $\overline{Z} = Y$. \\
By shrinking $Z$ around $z$, if needed, we can assume $Z$ lies in $ (\P^N_S \setminus V_+(x_{0}) )  \cong \A^N_S$.  So we end up with an affine open scheme $Z$ in $\A^N_S$ with fibre dimension 2 at the point $s$ and  whose projective closure is $Y$ with fibre dimension $3$ at the point $s$.
\end{example}

\begin{remark}
	Note that we need the base scheme to be at least of dimension 2 for the argument in previous example to work. In fact when the base is a Dedekind domain projective closure  of an equidimensional scheme will be equidimensional. 
\end{remark}

\begin{remark}\label{fiber-dim}
Throughout this section, $B$ will denote a pointed base scheme with a point $b\in B$. 
For a scheme $Z/B$, we use the phrase fibre dimension of $Z$ to denote $\dim(Z_b)$. 
\end{remark}

We now turn to the proof of our theorem. \cite[Theorem 2.1]{deshmukh} is proved for any scheme which is either smooth or a divisor in a smooth scheme. In fact, by applying induction on codimension the proof is valid for any $g \colon X\rightarrow S$ which is a local complete intersection at $x\in X$.

\begin{definition}
	\label{definition-lci-at-point}
	Let $g \colon X \rightarrow B$ be a morphism locally of finite type. The morphism $g$ is said to be a local complete intersection at a point $x \in X$ if there exists an open neighbourhood $U\subseteq X$ of $x$ such that there exists a commutative diagram
	\begin{center}
		\begin{tikzcd}
		X\arrow[dr] & U\arrow[l, hook']\arrow[r,"i", hook]\arrow[d]& \A^N_B\arrow[dl]\\
		& B&  
		\end{tikzcd}
	\end{center}
	where $i$ is a regular embedding. That is, $g|_U$ factors via a regular embedding $ i \colon U \hookrightarrow \A^N_B$. We say that a local complete intersection at $x$ has the length $r$ if the regular sequence associated to the regular embedding $i \colon U\hookrightarrow \A^N_B$ has the length $r$. In fact, using induction on the length of regular sequence and Krull principal ideal theorem, the length of the regular sequence is precisely the codimension of $U$ in $\A^N_B$ at $x$ (see \cite[02IE]{stacksradicial}).

	A morphism $g \colon X\rightarrow B$ locally of finite type is said to be a local complete intersection morphism if it is a local complete intersection at each point $x\in X$.
	
	Note that since we are working over a Noetherian base $B$, the above notion is independent of the choice of embedding $i$ (see \cite[068E, 063L]{stacksradicial}).
	
\end{definition}

The following lemma about the existence of generic hyperplanes will be used in the proof of Theorem \ref{ydense}. Let $B$ be a spectrum of a Noetherian ring.

\begin{lemma}\label{ver}
	Let $X$ be a closed subscheme in $\P^N_{B}$, with fibre dimension $n$ and $x$ be a point of $X$ lying over $b \in B$. Then there exist
	\begin{enumerate}
		\item a closed embedding of $X$ in $\mathbb{P}^{N'}_{B}$, for some $N' >0$.
		\item 	  $n$ hyperplanes $H_1, \cdots, H_n$ in $\mathbb{P}^{N'}_{B}$ missing $x$ 
		\item fibre dimension of $X \cap H_1 \cap \cdots \cap H_n$ is zero. (c.f. Remark \ref{fiber-dim})
	\end{enumerate}

\end{lemma}

\begin{proof}
By homogeneous prime avoidance lemma we can find hypersurfaces $K_1, \cdots, K_n$ in $\P^N_{B}$ satisfying conditions $(2)$ and $(3)$. Moreover by raising their degrees, if required, we can assume that all these hypersurfaces have the same degree $d$ and they are a part of a basis of $\Gamma(\P^N_B, O(d))$. Now taking degree $d$ Veronese embedding we obtain hyperplanes $H_1, \cdots, H_n$ in $\mathbb{P}^{{n+d \choose d}-1}_B$ which restrict to $K_i$'s in $\P^N_B$. Hence they satisfy the required conditions.
\end{proof}

The following intermediary lemma will be used repeatedly (see \cite[Lemma 10.1.4]{levine2006}).
\begin{lemma}\label{1-d}
	Let $X/B$ be an affine scheme and $x\in X$ which lies over $b$. Assume there exists an embedding $X\hookrightarrow    \A^N_B$ such that if $\overline{X}$ denotes the projective closure of $X$ in $\P^N_B$, then 
	$$ \dim(X_b) = \dim(\overline{X})_b = n.$$Then there exists 
	\begin{enumerate}
		\item an open neighbourhood $X_0$ of $x$ (in $X$), 
		\item a projective scheme $\tilde{X}$, 
		\item an open immersion $X_0\hookrightarrow \tilde{X}$ and 
		\item a projective morphism $\psi: \tilde{X}\rightarrow \P^{n-1}_B$
	\end{enumerate}
	such that all non-empty fibres of $\psi$ have dimension one.
\end{lemma}
\begin{proof}
	We follow the arguments given in \cite[Theorem 4.1]{kai2015}  (see also \cite[Theorem 10.1.4]{levine2006}).
	After possibly shrinking $B$, using Lemma \ref{ver} we can find $n$ hyperplanes $\Psi=\lbrace\psi_1,\ldots,\psi_n\rbrace$ which are a part of a basis of $\Gamma(\P^N_B,\sO(1))$ as a $B$-module. The choice is such that the vanishing locus of $\Psi$, denoted $V(\Psi)$, does not contain $x$ and  $\overline{X}\cap V(\Psi)$ is finite over $B$. Let $p \colon \tilde{\P^N}\rightarrow \P^N$ be the blowup of $\P^N$ along $V(\Psi)$, and $\tilde{X}$ the strict transform of $\overline{X}$. The rational map defined by $\Psi$ induces the map $\psi \colon \tilde{\P^N_B}\rightarrow \P^{n-1}_B$. Let $$X_0 =\overline{X}\setminus V(\Psi)\intersection X.$$ We have the following commutative diagram:
	\begin{center}
		\begin{tikzcd}[row sep=tiny]
		& \tilde{X} \arrow[dd] \arrow[r,hook,"cl."]&\tilde{\P^N_B}\arrow{dd}\arrow[r,"\psi"]&\P^{n-1}_B\\
		X_0 \arrow[ur,hook] \arrow[dr,hook] & \\
		& \overline{X}\arrow[r,hook,"cl."]&{\P^N_B}
		\end{tikzcd}
	\end{center}
	We claim that $\psi \colon \tilde{X}\rightarrow\P^{n-1}_B$ has fibre dimension one.
	To see this, choose any point $y\in\P^{n-1}_B$, and consider the composite $\Spec(\Omega)\overset{y}{\rightarrow}\P^{n-1}_B\rightarrow B$. Then, the fibre of $\psi$ over $y$ may be identified with a linear subscheme $V(y)$ of $\P^N_{\Omega}$, of dimension $N-n+1$. Furthermore, $V(y)$ contains the base change $V(\Psi)_{\Omega}$, which has dimension $N-n$, by construction. Again due to the construction, the intersection $V(y)\cap \overline{X}\cap V(\Psi)_{\Omega}$ is finite in $\P^N_{\Omega}$. This means that $V(y)\cap \overline{X}$ has dimension $1$ in the projective space $V(y)$.

	Further note that for $x\in V(\Psi)$, $p^{-1}(x)\simeq\P^{n-1}$. Also, the exceptional divisor of $\tilde{X}$ is an irreducible subscheme. Therefore, for any point $x\in V(\Psi)\cap X$, the fibre $\tilde{X}_b$ is an irreducible subscheme of $\P^{n-1}$ of dimension $n-1$. Therefore, $p^{-1}(\overline{X})=\tilde{X}$, so that $p \colon \psi^{-1}(y)\cap \tilde{X}\rightarrow V(y)\cap \overline{X}$ is a bijection. Thus, $\psi \colon \tilde{X}\rightarrow \P^{n-1}_B$ has 1-dimensional fibres.
\end{proof}

Now we prove Theorem \ref{ydense}.

\begin{proof}[Proof of Theorem \ref{ydense}] 
	The proof proceeds by double induction: on the dimension of the fibre and the length of the local complete intersection at $x$ (see Definition \ref{definition-lci-at-point}).\\
	
	\noindent \underline{Step 0}: \textbf{Setting up the double induction.} If length is zero, then Zariski locally, $X\simeq \A^N_B$	 for some $N$. In this case the proof is trivial. So, assume that the length is $1$. That is, Zariski locally, $X$ can be embedded as a hypersuface in $\A^N_B$.
	
	We will now prove the result for this case. After proving the length $1$ case, we will apply induction on the length to get the result for any local complete intersection.
	
	Fixing the length to be $1$, we proceed by induction on the dimension of the fibre $X_b$. The case $n=0$ follows from a version of Hensel's lemma.\\
	
	\noindent \underline{Step 1}: \textbf{Expressing as a relative curve over $\P^{n-1}_B$.}  Zariski locally, we can write $X$ as a hypersurface in $\A^{n+1}_B$. Let $\overline{X}$ denote its reduced closure in $\P^N_B$. By Lemma \ref{lemma-fibre-dimension-projective-closure}, $\dim (\overline{X})_b=n$. We can now apply Lemma \ref{1-d}, to get a projective morphism $\psi \colon \tilde{X}\rightarrow \P^{n-1}_B$ with 1-dimensional fibres. Now set $T= \P^{n-1}_B$ and $t=\psi(x)$.\\  
	
	\noindent\underline{Step 2}: \textbf{Constructing a finite morphism to an open subscheme of $\P^1_T$.}  Choose any projective embedding $\tilde{X} \hookrightarrow \P^{N_2}_T$. Let $(\tilde{X})_t$ and $(X_0)_t$ denote the fibres over $t$ of $\tilde{X}$ and $X_0$ respectively. Recall from Lemma \ref{1-d} that $X_0$ is an open neighbourhood of $x$. Choose a hypersurface $H_t \subseteq \P^{N_2}_t$ satisfying the following three conditions:
	\begin{enumerate}
		\item $x \in H_t$ (if $x$ is a closed point in $(X_0)_t$) 
		\item $(\tilde{X})_t$ and $H_t$ meet properly in $\P_t^{N_2}$
		\item $H_t$ does not meet $\overline{({X_0})_t}\setminus{({X_0})_t}$.
	\end{enumerate}
	Let $T^h\rightarrow T$ be the Henselisation of $T$ at $t$ and $H\subseteq \P^{N_2}_{T^h}$ a flat family of hypersurfaces over $T^h$ which restricts to $H_t$ over $t$.
	
	Now we restrict to a suitable affine Nisnevich neighbourhood of $T$, which we again denote by $T$ (and  base change everything to $T$). Using the hypersurface $H$ and structure of finite algebras over a henselian ring, we can choose a Cartier divisor $\sD$ which fits into the following diagram.
	\begin{center}
		\begin{tikzcd}[row sep=tiny,column sep=large]
		& \tilde{X} \arrow{r}{projective}[swap]{1-dim}&T\arrow[r,"Nis"]&\P^{n-1}_B\\
		X_0 \arrow[ur,hook] & \\
		& \sD\arrow[uu]\arrow[ul,hook,"{Cartier. div}"]\arrow{uur} [swap]{finite}
		\end{tikzcd}
	\end{center}
	For sufficiently large $m$ we can find a section $s_0$ of $\Gamma(\tilde{X},\sO_{\tilde{X}}(m\sD))$ which maps to a nowhere vanishing section of $\Gamma(\sD,\sO_{\sD})$. Let $s_1 \colon \sO_{\tilde{X}} \ra \sO_{\tilde{X}}(m\sD)$ be the canonical inclusion. Since the zero-loci of $s_0$ and $s_1$ are disjoint, we get a map
	$$f=(s_0,s_1) \colon\tilde{X} \ra \P^1_{T}.$$
	
	Since the quasi-finite locus of a morphism is open, shrink $T$ around $t$ such that $\sD$ is contained in the quasi finite locus of $f$ after the base change. Let $X_0'$ be the quasi-finite locus of the base change.
	
	\begin{center}
		\begin{tikzcd}[column sep=large]
		& f^{-1}(\infty_T)=\sD \arrow[d] \arrow[r,hook]&X_0'\arrow{d}[swap]{{quasi-finite}}\arrow[r,hook]&\tilde{X}\arrow[dl,"f"]\\
		& \infty_T\arrow[r,hook]&{\P^1_T}
		\end{tikzcd}
	\end{center}
	Then the subset $W=f(\tilde{X}\setminus X_0')\subseteq \P^1_T$ is proper over $T$ and is contained in $\P^1_T\setminus H_t=\A^1_T$. Hence, it is finite over $T$. The map $\tilde{X}\setminus f^{-1}(W)\ra \P^1_T\setminus W$, being proper and quasi-finite, is finite. By condition $(1)$, we see that $\tilde{X}\setminus f^{-1}(W)$ contains $x$.\\
	
	\noindent\underline{Step 3}: \textbf{Constructing the required neighbourhoods.} Now by induction there exist Nisnevich neighbourhoods $B_1\ra B$ and $T_1 \ra T$ such that the projective compactification $T_1\ra \overline{T_1}$ is fibre-wise dense in the union of $n$-dimensional irreducible components over $B_1$. Take a factorization of $f$ of the form $$\tilde{X} \hookrightarrow \P^{N_3}_{T_1}\times_{T_1}\P^1_{T_1} \ra \P^1_{T_1}.$$ Let $\overline{X_1}$ denote the reduced closure of $\tilde{X}$ in $\P^{N_3}_{\overline{T_1}}\times_{\overline{T_1}}\P^1_{\overline{T_1}}$. We get the following diagram where every square is Cartesian
	\begin{center}
		\begin{tikzcd}[row sep=large]
		X_2:=\tilde{X}\setminus f^{-1}(W) \arrow[d,hook] \arrow[r,hook]&\tilde{X}\arrow[d,hook]\arrow[r,hook]&\overline{X_1}\arrow[d,hook]\\
		\P^{N_3}_{T_1}\times_{T_1}(\P^1_{T_1}\setminus W) \arrow[r,hook]\arrow{d}&\P^{N_3}_{T_1}\times_{T_1}\P^1_{T_1} \arrow[r,hook]\arrow{d}&\P^{N_3}_{\overline{T_1}}\times_{\overline{T_1}}\P^1_{\overline{T_1}}\arrow{d}\\
		\P^1_{{T_1}}\setminus W \arrow[r,hook] &\P^1_{{T_1}}\arrow[r,hook]&\P^1_{\overline{T_1}}
		\end{tikzcd}
	\end{center}
	
	By Stein factorization we decompose the map $\overline{f_1}\colon \overline{X_1}\ra \P^1_{\overline{T_1}}$ as 
	$$ \overline{f_1} \colon \overline{X_1}\ra \overline{X_2}\xrightarrow{finite} \P^1_{\overline{T_1}},$$ where the first map has geometrically connected fibres. Since $\overline{f_1}$ is finite over the open set $\P^1_{T_1}\setminus W$, 	$\overline{X_2}\times_{\P^1_{\overline{T_1}}}(\P^1_{T_1}\setminus W)$ is isomorphic to $X_2 :=\tilde{X}\setminus f^{-1}(W) $. Since ${X_2}$ is open in $\overline{X_2}$, the fibre dimension of $\overline{X_2}$ is at least $n$. Combining this with the fact that $\overline{X_2}$ is finite over $\P^1_{\overline{T_1}}$, we conclude that the fibre dimension of $\overline{X_2}$ over $B_{1}$ is exactly $n$.
	
	We observe that since $T_{1}$ is fibrewise dense in the union of $n$-dimensional irreducible components of $\overline{T_{1}}$, so is $\P^1_{{T_1}}$ (in $\P^1_{\overline{T_1}}$). Also as $W$ is finite over $T_{1}$, $\P^1_{T_1}\setminus W$ is fibrewise dense in $\P^1_{T_1}$. Hence it is dense in the union of $n$-dimensional irreducible components of $\P^1_{\overline{T_1}}$.
	Now we claim that $X_2 $ intersects the fibre of $\overline{X_2}$ over any point $b_{1}$ of $B_{1}$. Let $X_{2}'$ be an $n$-dimensional irreducible component of the fibre $(\overline{X_2})_{{b}_{1}}$. Then the induced map $X_{2}' \rightarrow( \P^1_{\overline{T_1}})_{{b}_{1}}$ is a finite morphism of schemes of the same dimension. Hence it is a surjection to an irreducible component say, $U$ of $( \P^1_{\overline{T_1}})_{{b}_{1}}$. Further $\P^1_{T_1}\setminus W$ intersects $U$ by denseness. Taking the inverse image of its intersection with the irreducible component proves that $X_{2}$ intersects $X$.
	
	Since $\overline{X_2}$ is projective over $B_{1}$, we choose any embedding of it in a projective space $\P^{N}_{{B}_{1}}$. Then for the closed subscheme $\overline{X_2} \setminus X_2$ (with the reduced structure) there exists a hypersurface $H$ of $\P^{N}_{{B}_{1}}$ of degree, say $d$, containing $\overline{X_2} \setminus X_2$, not containing the point $x$ and such that $H_{{b}_{1}}$ intersects $({X_2})_{{b}_{1}}$ properly in $\P^{N}_{{b}_{1}}$. Hence by discussion in previous paragraph, $H_{{b}_{1}}$ also intersects $(\overline{X_2})_{{b}_{1}}$ properly. Replacing $X_{2}$ by $\overline{X_2} \setminus H$ and taking $d$ fold Vernose embedding we may assume $H$ to be $\P^{N-1}_{\infty}$. This gives us an embedding $$\overline{X_2} \setminus H \hookrightarrow \A^{N}_{{B}_{1}} = \P^{N}_{{B}_{1}} \setminus \P^{N-1}_{\infty}, $$ with the required properties.\\
	
	\noindent\underline{Step 4}: \textbf{Dealing with length r.}  Let the length of the regular sequence at $x$ be $r$. Then Zariski locally, $X$ is given by a regular sequence $(f_1,\ldots, f_r)$ in $\A^{N}_B$. Now, consider the affine scheme $V:=\Spec B[X_1,\ldots, X_{n+r}]/(f_1,\ldots,f_{r-1})\subseteq \A^N_B$. Then in a neighbourhood of the point $x$, $X$ can be thought of as a divisor in $V$ cut out by the function $f_r$. Note that $V$ has the length $r-1$ at $x$. Hence, by induction, the Theorem holds for $V$. As $X$ is a divisor in $V$ (cut out by the function $f_r$), applying Lemma \ref{lemma-fibre-dimension-projective-closure}, we get that $\dim (\overline{X})_b=n$. Here $\overline{X}$ is the closure of $X$ inside the compactification $\overline{V}$ of $V$ which we get by induction on the length of the regular sequence.
	
	Then by Lemma \ref{1-d}, we have a commutative diagram,
	\begin{center}
		\begin{tikzcd}[row sep=tiny]
		& \tilde{X} \arrow[dd] \arrow[r,hook,"cl."]&\tilde{\P^N_B}\arrow{dd}\arrow[r,"\psi"]&\P^{n-1}_B\\
		X_0 \arrow[ur,hook] \arrow[dr,hook] & \\
		& \overline{X}\arrow[r,hook,"cl."]&{\P^N_B}
		\end{tikzcd}
	\end{center}
	such that all non-empty fibres of $\psi$ are $1$-dimensional.
	
	Further as in Step 2 above, Nisnevich locally on $X$ we obtain a morphism, $\phi \colon X \rightarrow \P^{1}_{T}$, where $T$ is a Nisnevich neighbourhood of $\P^{n-1}$. 
	Since $T$ is a smooth $B$-scheme, the theorem holds for $T$. 
	The remaining proof is the same as in Step 3.
\end{proof}

The following lemma is used in the proof of Theorem \ref{ydense} above for applying induction on the length of the local complete intersection at a point.

\begin{lemma}\label{lemma-fibre-dimension-projective-closure}
	Let $X$ be a finite type scheme over $B$ such that the conclusion of Theorem \ref{ydense} holds. Assume that $\dim X_b=n+1$. Let $Y\subseteq X$ be a divisor and let $Y'\subseteq X'$ be its inverse image in $X'$. Denote by $\overline{Y'}\subseteq \overline{X'}$ the closure of $Y'$ in $\overline{X'}$. Then $\dim Y'_b=\dim (\overline{Y'})_b$.
\end{lemma}

\begin{proof}
As the conclusion of Theorem \ref{ydense} holds for $X$, after passing to the Nisnevich neighbourhoods $X'$ and $B'$, we may assume that $X=X'$, $B=B'$ and $Y=Y'$.  We have an embedding $X \rightarrow \A_B^N$ such that $X_b$ is dense in the $(n+1)$-dimensional components of $(\overline{X})_b$. Thus, we have inclusions

\begin{enumerate}
	\item $Y_b \subseteq X_b \subseteq (\overline{X})_b$,
	\item $Y_b \subseteq (\overline{Y})_b \subseteq (\overline{X})_b$.
\end{enumerate}

Since $Y$ is a divisor in $X$, it is given by a single equation $f$, say. Let $F$ be the homogenisation of $f$ (inside the projective closure $\overline{X}$). Then $F$ cuts out the subscheme $(\overline{Y})_b$ in $(\overline{X})_b$ in the fibre over $b\in B$. In this case, the dimension of $(\overline{Y})_b$ is at most one less than the dimension of $(\overline{X})_b$.\\ 
We now have two cases to consider depending on whether or not $Y_b$ contains $(n+1)$-dimensional components of $X_b$.\\

\noindent\underline{Case 1}: $\dim Y_b=\dim X_b=n+1$. In this case, $Y_b$ contains irreducible components of $X_b$ of dimension $n+1$. As $\dim (Y_b) \leq \dim (\overline{Y})_b\leq \dim (\overline{X})_b$, we must have that $\dim Y_b=\dim (\overline{Y})_b$, since $X_b$ is dense in the $(n+1)$-dimensional components of $(\overline{X})_b$.\\

\noindent\underline{Case 2}: $\dim Y_b =n = \dim (X_b)-1 $. In this case, we analyse the two possibilities for the dimension of $\overline{Y}_b$.

If it is exactly one less than $\dim \overline{X}_b$, we are done, since dimension of $(\overline{X})_b$ is $n+1$.

If not, then we have $(\overline{Y})_b = n+1$. This means that $\overline{Y}_b$ contains an irreducible component $Z\subseteq (\overline{X})_b$ of dimension $n+1$. As $X_b$ is dense  in the $(n+1)$-dimensional components of $(\overline{X})_b$, the generic point of $Z$ must lie in $X_b$ implying that it also lies in $Y_b$. 
This says that $\dim (X_b)=\dim (Y_b)$, which is a contradiction.

\end{proof}

\section{Applications}

In this section, we present an application of Theorem \ref{ydense} to the existence of finite maps to $\A^n$, Nisnevich locally (Theorem \ref{fer}).

\begin{remark}\label{vers}
	Let $B$ be the spectrum of a Noetherian ring of finite dimension. Let $g \colon X\rightarrow B$ be a finite type scheme over $B$. Let $x\in X$ be a point lying over a point $b\in B$ such that $g$ is a local complete intersection at $x$ and $\dim(X_b) = n$.  Let $  X \hookrightarrow \overline{X} \xrightarrow{\phi} \P^N_S$ be the embedding satisfying the fibrewise denseness condition stated in the Theorem \ref{ydense}. Note that the stated fibrewise denseness condition is equivalent to the fact that $ V_+(x_0) := H_{\infty}$ in $\P^N_S = $ Proj $S[\ x_0, \cdots ,x_n]\ $  cuts down the fibre dimension of the closed subscheme $(\overline{X})_b$. Then composing $\phi$ by any degree $d$ Veronese embedding of $\P^N_S$ in $\P^{N'}_S$   preserves the fibrewise denseness of the new embedding $X \hookrightarrow \overline{X} \xrightarrow{\psi} P^{N'}_S$.
\end{remark}

Let $B= Spec(R)$ with $\A^{n}_B = R[x_1,\ldots,x_n]$. Let $E$ be the $R$-span of $\{x_1,\ldots,x_n\}$ and consider $\sE:=\underline{\Spec}(Sym^{\bullet}E^{\vee})$ (note that $\sE(R)=E$). For any integer $d > 0$ and any $R$-algebra $A$, $\sE^d(A)$ parametrises all linear morphisms $v= (v_1,\ldots,v_{d}) \colon \A^{n}_T\rightarrow \A^{d}_T$, where $T= Spec(A)$. Considering $\A^n_B \hookrightarrow \P^{n}_B = $ Proj $R[x_0,\ldots,x_n]$, as a distinguished open subscheme $D(x_{0})$, we extend such a linear morphism to a rational map $ \overline{v} \colon \P^{n}_B \dashrightarrow  \P^{d}_B$ whose locus of indeterminacy $L_{v}$ is given by the vanishing locus of $v_1,\ldots,v_{d}$ and $x_0$, $V_{+}(x_0,v_1,\ldots,v_{d})  \subseteq \P^{n}_B$ (We will use this notation throughout what follows). Given any closed subscheme $Y$ in $\A^n_B$, we denote by $\overline{Y}$ its projective closure in $ \P^{n}_B $. For the following lemma we refer to \cite[Lemma 2.3]{schmidt2018}.

\begin{lemma}(see \cite[Lemma 2.3]{schmidt2018})\label{shafar}
	In the setting of the previous paragraph if $L_{v} \cap \overline{Y} = \emptyset$, then $\overline{v} \colon \overline{Y} \rightarrow  \P^{d}_B$ and $v \colon Y \rightarrow \A^d_B$ are finite maps.
\end{lemma}

 \begin{proposition}
Let $B$ be the spectrum of a Noetherian ring of finite dimension. Let $g \colon X\rightarrow B$ be a finite type scheme over $B$. Let $x\in X$ be a point lying over a point $b\in B$ such that $g$ is a local complete intersection at $x$ and $\dim(X_b) = n$.  Let $k$ be the residue field at the point $b \in B$. Let $\overline{X}$ be the  closure of $X$ in $\P^N_B$. Then there exist $\nu = (\nu_1, \cdots, \nu_n)$ such that $L_{\nu} \cap (\overline{X})_b = \emptyset $.
	 \end{proposition}
 
 \begin{proof}
 	Denote by $H_{\infty}$ the hyperplane $V_+(x_0)$ in $\P^N_k$. We have $dim (\overline{X_b} \cap H_{\infty}) = n-1$. Hence by Theorem \ref{ydense}, $dim ((\overline{X})_b \cap H_{\infty}) = n-1$. As $k$ need not be infinite we have hypersurfaces---not necessarily hyperplanes --- $H_1, \cdots H_n$ in $\P^N_k$ such that $ (\overline{X})_b \cap \tilde{H} = \emptyset $, where $\tilde{H} = H_{\infty} \cap H_1 \cap \cdots \cap H_n$. Without loss of generality assume that these hypersurfaces have degree $d $. Take degree $d$ Veronese embedding and consider the closed embedding $(\overline{X})_b \hookrightarrow \P^{N'}_k$. Now the proposition follows from Remark \ref{vers}.
 \end{proof}

We state the next theorem whose proof is now exactly the same as Proposition 3.9 of \cite{deshmukh}. Theorem \ref{fer} is a crucial step in proving Gabber presentation lemma over a general base.  

\begin{theorem}\label{fer}
	Let $X=Spec(A)/B$ be a smooth, equi-dimensional, affine, irreducible scheme of relative dimension $n$ and $B$, spectrum of a Henselian local ring. Let $Z=Spec(A/f)$, $z$ be a closed point in $Z$ lying over $b\in B$, where $f$ is such that $\dim (Z_b) < \dim (X_b)$. Then there is an open subset $\Omega\subseteq \sE^{d}$ with $\Omega(R)\neq \emptyset$ such that for all $\Psi \in \Omega(R)$, $\Psi_{|Z} \colon Z\ra \A^{n-1}_S$ is finite.
\end{theorem}

\begin{remark}
	A weaker version of Gabber presentation lemma for essentially smooth Henselian schemes over a local Henselian base has been proved in \cite{druzhinin2019}. It turns out that for many applications in $\A^1$-homotopy theory this version of Gabber presentation lemma is sufficient. For example, this can be used to prove Shifted $\A^1$-Connectivity over a general base (\!\! \cite{druzhinin2019}).
\end{remark}

\bibliographystyle{alphanum}
\bibliography{biblio}

\end{document}